\documentclass{amsart}
\usepackage{amssymb,amsmath,amstext}
\usepackage{amsmath,amssymb,amscd,amsmath,amsfonts,pb-diagram,graphicx}
\usepackage[latin1]{inputenc}
\usepackage{a4}
\newtheorem{corollary}{Corollary}[section]

\newtheorem{definition}{Definition}[section]
\newtheorem{lemma}{Lemma}[section]

\def\proof{\par{\bf Proof}. \ignorespaces}
\def\qedsymbol{\vbox{\hrule\hbox{%
                     \vrule height1.3ex\hskip0.8ex\vrule}\hrule}}
\def\endproof{\qquad\qedsymbol\medskip\par}

\newtheorem{proposition}{Proposition}[section]
\newtheorem{remark}{Remark}[section]
\newtheorem{theorem}{Theorem}[section]

\newcommand{\remove}[1]{}
\def\al{\alpha}
\def\be{\beta}
\def\ga{\gamma}

\def\de{\delta}

\def\ep{\varepsilon}
\def\ze{\zeta}

\def\th{\theta}

\def\ka{\kappa}
\def\la{\lambda}

\def\si{\sigma}

\def\Si{\Sigma}

\def\C{{\mathbb C}}

\def\P{{\mathbb P}}

\def\U{{\mathbb U}}

\def\CH{{\mathcal H}}

\newcommand {\ra}{\rightarrow}

\newcommand {\diag}{\mathrm{diag}}
\newcommand {\rank}{\mbox{rank }}

\renewcommand {\ker}{\mbox{ker }}
\newcommand{\invder}[2]{D#1(#2)|_{#2^\perp}^{-1}}

\newcommand{\Hd}{\mathcal H_{(d)}}

\begin{document}
\title[Step Size Selection for Homotopy Methods]{Adaptive Step Size Selection for Homotopy Methods to Solve Polynomial Equations}
\author{Jean-Pierre Dedieu}
\address {Institut de Mathématiques de Toulouse, Université Paul Sabatier, 31069 Toulouse Cedex 9, France}
\email{jean-pierre.dedieu@math.univ-toulouse.fr}
\urladdr{http://www.math.univ-toulouse.fr/~dedieu/}
\author{Gregorio Malajovich}
\address{Instituto de Matemática, Universidade Federal do Rio de Janeiro.
Caixa Postal 68530 Rio de Janeiro RJ 21941-909 Brasil.}
\email{gregorio.malajovich@gmail.com}
\urladdr{http://www.labma.ufrj.br/~gregorio}
\author{Michael Shub}
\address {CONICET, IMAS, Universidad de Buenos Aires, Argentina and CUNY Graduate School, New York, NY, USA.}
\email{shub.michael@gmail.com}
\urladdr{http://sites.google.com/site/shubmichael/}
\thanks{This research was funded by MathAmsSud grant {\em Complexity}. Michael Shub was partially supported by a CONICET grant.
Gregorio Malajovich is partially supported by CNPq, FAPERJ and CAPES (Brasil).}

\date{June 6,2011}

\begin{abstract}

Given a $C^1$ path of systems of homogeneous polynomial equations $f_t$, $t \in [a,b]$  and an approximation $x_a$ to a zero $\zeta_a$ of the initial system $f_a$, we show how to adaptively choose the step size for a Newton based homotopy
method so that we approximate the lifted path $(f_t,\zeta_t)$ in the space of $(problems, solutions)$ pairs.
 The total number of Newton iterations is bounded in terms of the length of the lifted path in the condition metric.
\end{abstract}
\maketitle

\section{Introduction}

\noindent Let us denote by
$\CH_{(d)}$ the vector space of homogeneous polynomials systems
$$f : \C^{n+1} \rightarrow \C^n,$$
$f=(f_1, \ldots , f_n)$ in the variable $z = (z_0, z_1, \ldots , z_n)$, with degree $(d) = (d_1, \ldots , d_n)$, so that $f_i$ has degree $d_i$.

Given a $C^1$ path of systems $t \in [a,b] \ra f_t \in \CH_{(d)}$, and  a zero $\zeta_a$ of the initial system $f_a$, under very general conditions, the path $t\rightarrow f_t$ can be lifted to a $C^1$ path $t\rightarrow (f_t,\zeta_t)$ in the solution variety 
$$\widehat{V}  = \left\{ (f, \ze) \in \CH_{(d)} \times \C^{n+1} \ : \ f(\ze ) = 0 \right\} .$$
If we make the additional hypothesis that $\frac{d \zeta_t}{d t}$ is orthogonal to $\zeta_t$ this path is unique.

Now, given a sufficiently close approximation $x_a$ to the zero $\zeta_a$ of the initial system $f_a$, predictor-corrector methods based on Newton's method may approximate the lifted path $(f_t,\zeta_t)$ by a finite number of pairs $\left( f_{t_i}, x_i \right) \in \CH_{(d)} \times \C^{n+1}$, $0 \leq i \le k$. These algorithms are designed as follows: first the interval $[a,b]$ is discretized by a finite number of points $a= t_0 < t_1 < \ldots < t_k = b$, then a sequence $(x_i)$ is constructed recursively by 
$$x_0 = x_a \mbox{ and } x_{i+1} = N_{f_{t_{i+1}}}(x_i)$$ 
where $N_{f_{t}}$ is the projective Newton operator associated with the system $f_t$. The complexity of such algorithms is measured by the size $k$ of the subdivision $(t_i)$. If we make a good choice for $(t_i)$ then $k$ is small and, for each $i$, $x_i$ is an approximate zero of $f_{t_i}$ associated with $\ze_{t_i}$.

The complexity of such algorithms has been related by Shub-Smale in \cite{Bez1} to the length $l$ of the path $(f_t)$ and to the condition number of the path 
$(f_t, \ze_t)$: $\mu = \max_{a \le t \le b}\mu(f_t, \ze_t)$. The condition number measures the size of the first order variations of the zero of a polynomial system in
terms of the first order variations of the system. For $(f,\ze) \in V$ it is given by
$$\mu ( f, \ze) = \left\| f \right\| 
\left\|
\left( Df(\ze)\left|_{\ze^\perp}\right. \right)^{-1} 
\diag \left(\sqrt{d_i}\left\| \ze \right\|^{d_i - 1} \right)
\right\|$$
or $\infty$ when $\rank Df(\ze)\left|_{\ze^\perp}\right. < n$.
We extend this definition to any pair $(f,z) \in \CH_{(d)} \times \C^{n+1}$ by the same formula.
Shub-Smale in \cite{Bez1} give the bound
$$k \leq C D^{3/2}l \mu^2,$$
where $C$ is a constant, and $D = \max d_i$.

A more precise estimate is given in Shub \cite{shu}. The author proves that we can choose the step size in predictor-corrector methods so that the number of steps sufficient to approximate the lifted path is bounded in terms of the length of the lifted path in the condition metric: 
$$L = \int_a^b \left(\left\| \frac{df_t}{dt} \right\|_{f_t}^2 + \left\| \frac{d\ze_t}{dt} \right\|_{\ze_t}^2 \right)^{1/2}\mu(f_t, \ze_t) dt.$$
We can find such a subdivision $(t_i)$, $0 \le i \le k$, with
$$k \leq C D^{3/2} L.$$
Its construction is given by 
$$t_0 = a, \mbox{ and } \int_{t_i}^{t_{i+1}} \left(\left\| \frac{df_t}{dt} \right\|_{f_t} + \left\| \frac{d\ze_t}{dt} \right\|_{\ze_t} \right) dt = 
\frac{C}{D^{3/2}\mu ( f_{t_i}, \ze_{t_i})}$$
but, in this paper, some universal constants are not estimated, and no constructive algorithm is given.

In the algorithm we present below, we compute $t_{i+1}$ from $t_i$ so that at least one of the quantities 
  $$\int_{t_i}^{t_{i+1}}\left\| \frac{df_t}{dt} \right\|_{f_t}  dt $$
  and 
  $$\int_{t_i}^{t_{i+1}}  \left\| \frac{d\ze_t}{dt} \right\|_{\ze_t} dt$$
  increases of a given fraction of 
  $$\frac{1}{D^{3/2}\mu ( f_{t_i}, \ze_{t_i})}.$$
Moreover, to allow approximate computations, we introduce a tolerance parameter $\ep$. 
This algorithm reflects the geometrical structures used in \cite{shu}. This structure is based on a Lipschitz-Riemannian metric defined in the solution variety $V$ by 
$$\left\langle .,. \right\rangle_{V,(f,\ze)}\mu(f,\ze)^2$$
where $\left\langle .,. \right\rangle_{V,(f,\ze)}$ is the Riemannian metric in $V$ inherited from the usual metric on $\P({\CH}_{(d)})\times \P(\C^{n+1})$. This condition metric is studied in more details in Beltr\'an-Dedieu-Malajovich-Shub \cite{bel1} and \cite{bel1bis}, Beltr\'an-Shub \cite{bel2} and \cite{bel3}, and in Boito-Dedieu \cite{boi}. 

\vskip 5mm
\fbox{
\begin{minipage}{\textwidth}
\noindent
\begin{trivlist}
\item {\sc Algorithm} {\bf Homotopy}
\item {\sc Input:} $(f_t)_{t \in [a,b]}$, $0 \ne x_0 \in \C^{n+1}$, $0 < \epsilon \leq 1/20$.
\item {\sc Output:}
\subitem An integer $k \geq 1$,
\subitem A subdivision $a= t_0 < t_1 < \ldots < t_k = b$,
\subitem A sequence of nonzero points $x_i \in \C^{n+1}$, $0 \le i \le k$.
\item {\sc Algorithm:}
\item $i \leftarrow 0$; $t_0 \leftarrow a$;
\item {\bf Repeat}
\subitem {\bf Find} $s \in [t_i, b]$ so that
\begin{equation}\label{s1}
\frac{20 \epsilon^2}{5 D^{1/2} \mu(f_{t_i},x_i)}
\le
\int_{t_i}^s \left\|  \frac{df_t}{dt}  \right\|_{f_t} dt
\le
\frac{\epsilon}{5 D^{1/2} \mu(f_{t_i},x_i)} .
\end{equation}
In case there is no such $s$, make $s=b$.\\
\subitem
{\bf Find} $s' \in [t_i,b]$ so that for all $\sigma \in [t_i, s']$,
\begin{equation}\label{s2}
\frac{20 \epsilon^2}{5 D^{3/2} \mu(f_{t_i},x_i)}
\le
\phi_{t_i, \sigma}(x_{i})
\le
\frac{\epsilon}{ 5 D^{3/2} \mu(f_{t_i},x_i)}
\end{equation}
where $$\phi_{t_i, \sigma}(x_{i}) = \lVert x_i \rVert^{-1} \lVert Df_{t_i}(x_i)|_{x_i^\perp}^{-1}\left(f_{t_i}(x_i) - f_\sigma(x_i)\right) \rVert .$$
In case there is no such $s'$, make $s'=b$.
\subitem $t_{i+1} = \min( s, s')$;
\subitem {\bf Find} $0 \ne x_{i+1} \in \mathbb C^{n+1}$ such that
\begin{equation}\label{s3}
d_R\left(N_{f_{t_{i+1}}}(x_i), x_{i+1}\right) < \frac{4\epsilon^2}{5D^2 \mu(f_{t_i},x_i)^2} ;
\end{equation}
\subitem $i \leftarrow i+1$;
\item {\bf Until} $t_i = b$.
\item {\bf Set $k \leftarrow i$}
\end{trivlist}
\end{minipage}
}
\vskip 5mm

This algorithm has the following properties: 

\begin{theorem}\label{main-homotopy} Given $0< \epsilon \leq 1/20$, a $C^1$ homotopy path $(f_t)_{t \in [a,b]}$ in $\CH_{(d)}$, and an initial point $0 \ne x_0 \in \C^{n+1}$ satisfying
\[
\frac{D^{3/2}}{2} \mu (f_a, x_0) \beta_0(f_a, x_0) < \frac{\epsilon^2}{2},
\]
then:
\begin{itemize}
\item[1.] $x_0$ is an approximate zero of $f_a$ with associated nonsingular zero $\zeta_a$,
\item[2.] Let $(f_t, \zeta_t)_{t\in[a,b]}$ be a continuous lifting of $(f_t)_{t \in [a,b]}$ in the solution variety initialized at $(f_a, \zeta_a)$. If the 
condition length $L$ is finite, then:
\item[2.a.] The algorithm {\bf Homotopy} with input $(\epsilon, (f_t), x_0)$ stops after at most 
$$k = 1+ 0.65 D^{3/2} \ep^{-2} L$$ 
iterations of the main loop,
\item[2.b.] For each $i=1 \ldots  k$, $x_i$ is an approximate zero of $f_{t_i}$ with associated zero $\zeta_{t_i}$.
\end{itemize}
\end{theorem}

\begin{remark} \begin{itemize} \item This algorithm is robust: it is designed to allow approximate computations. 
\item For $\ep = 1/20$, the computations in \ref{s1} and \ref{s2} have to be exact. 
\item The hypothesis { \it ``the condition length $L$ is finite''} holds generically for $C^1$ paths in the solution variety $V$ and, consequently, Theorem \ref{main-homotopy} holds for ``almost all'' inputs $(f_t)_{t \in [a,b]}, x_0$.
\item We reach the same complexity as in Shub \cite{shu}. 
\item We require the regularity $C^1$ for our path in the space of systems but, in fact, an absolute continuous path $t \in [a,b] \ra f_t \in \Hd$ is sufficient to define the concepts of length and condition length and to prove Theorem \ref{main-homotopy}. The proofs given in section \ref{sec-homotopy} are still valid.
\end{itemize}
\end{remark}

Let us now mention other approaches to the construction of {\it ``convenient subdivisions''}. 

A practical answer consists in taking $t_{i+1} = t_i + \delta t$ for an arbitrary $\delta t > 0$. If $x_{i+1}$ fails (resp. succeeds) to be an approximate zero of $f_{t_{i+1}}$ then we take $\delta t / 2$ (resp. $2\delta t$) instead of $\delta t$; see  Li \cite{TYL}. With such an algorithm we may jump from a lifted path $t \ra (f_t,\ze_t) \in V$ to another one $t \ra (f_t,\zeta'_t) \in V$. Even if, for each $i$, $x_i$ is an approximate sero of $f_{t_i}$, we cannot certify that the sequence $(f_{t_i}, x_i)$ approximates the path $(f_t,\ze_t)$.

In Beltr\'an \cite{beltran}, the author presents an algorithm to construct a certified approximation of the lifted path. It requires a $C^{1 + Lip}$ path in the space of systems, and it has an additional multiplicative factor in the number of steps given in \cite{shu}. This extra factor is unbounded for the class of $C^1$ paths considered here. 

Beltr\'an's algorithm is studied in more detail in Beltr\'an-Leykin \cite{beltran-leykin}, which contains implementations and experimental results.  

Another important problem, which is not considered here, is the choice of both the homotopy path $(f_t)$, $a \le t \le b$, and the initial zero $\ze_a$. 
Classical strategies are described in Li \cite{TYL}, and Sommese-Wampler \cite{som}.

A conjectured ``good choice'' (see Shub-Smale \cite{Bez5}) is the system 
\[
f_a(z_0, z_1, \ldots , z_n) = 
\begin{cases}
d_1^{\frac{1}{2}}z_0^{d_1-1}z_1=0,\\
\cdots\\
d_n^ {\frac{1}{2}}z_0^{d_n-1}z_n=0,
\end{cases}
\;\;\;\ze_a = (1,0,\ldots,0),
\]
and a linear homotopy connecting this initial system to the target system $f_b$. See \cite{Bez5} for a precise statement. This conjecture is still unproved. In Shub-Smale \cite{Bez5} an adaptive algorithm is given for linear homotopies whose number of steps is bounded by the estimate in Shub-Smale \cite{Bez1}. The algorithm we present here is a version of that algorithm adapted to the new context of length in the condition metric.

Beltr\'an-Pardo \cite{Beltran-Pardo} use a linear homotopy and Beltr\'an's strategy for the choice of the subdivision. They get, for a random choice of 
$(f_a, \ze_a)$, an average running time $O\Tilde{ }(N^2)$ where $N$ is the size of the input. 

Buergisser-Cucker in \cite{Buergisser-Cucker} define an explicit algorithm, called ALH, based on the linear homotopy and a certain adaptive construction for the subdivision. They obtain the complexity
$$217 D^{3/2} d_\P(f_a, f_b) \int_a^b \mu(f_t, \ze_t)^2dt$$
which is not as sharp as the estimate based on the condition length given in \cite{shu} or to our own estimate. Then, we cite the authors, {\it ``ALH will
serve as the basic routine for a number of algorithms computing zeros of
polynomial systems in different contexts. In these contexts both the input
system $f_b$ and the origin $(f_a, \ze_a)$ of the homotopy may be randomly chosen''.} See this manuscript for a more detailled description.

Our feeling, based on a series of papers on the condition metric: \cite{bel1}, \cite{bel1bis}, \cite{bel2}, \cite{bel3}, \cite{boi}, is that a good choice for the homotopy path $(f_t)$ and the initial zero $\ze_a$ will induce a lifted path $(f_t, \ze_t)$ close to the condition geodesic connecting $(f_a, \ze_a)$ to 
$(f_b, \ze_b)$. We are far from accomplishing this task. 

Our paper is organized as follow. Section \ref{sect-2} recalls the geometric context and contains the main definitions. In section \ref{sect-3} we study the variations of the condition number $\mu(f,x)$ when we vary both the system $f$ and the vector $x$. The main difficulty is to estimate universal constants which are already present in 
many papers (Shub-Smale \cite{Bez5}, Shub \cite{shu} for example) but which are not given explicitely. Such explicit constants are necessary to design an explicit algorithm. Section \ref{sec-alpha-th} contains, in the same spirit, explicit material about projective alpha-theory. The last section is devoted to the proof of 
Theorem \ref{main-homotopy}.

\section{Context and definitions}\label{sect-2}

\subsection{Definitions}
We begin by recalling the context. For every positive integer $l\in{\mathbb N}$, let ${\mathcal H}_l\subseteq {\mathbb C}[x_0,\ldots,x_n]$, $n \ge 2$, be the vector space of homogeneous polynomials of degree $l$.  For $(d)=(d_1,\ldots,d_n)\in{\mathbb N}^n$, let ${\mathcal H}_{(d)}=\prod_{i=1}^n {\mathcal H}_{d_i}$ be the set of
all systems $f=(f_1,\ldots,f_n)$ of homogeneous polynomials of respective degrees $\deg(f_i)=d_i,~1\leq i\leq n$. So $f: \mathbb{C}^{n+1}\rightarrow \mathbb{C}^n.$ We denote by $D:=\max\{d_i:1\leq i\leq n\}$ the maximum of the degrees, and we suppose $D \geq 2$.

The {\it solution variety} $\widehat{V}$ is the set of points $(f, \ze) \in \CH_{(d)} \times \C^{n+1}$ with $f(\ze)=0.$
Since the equations are homogeneous, for all $\lambda_1,\lambda_2 \in \mathbb{C} \setminus \{0\}$, $\lambda_1f(\lambda_2 \ze) = 0$ if and only if
$f(\ze)=0.$ So $\widehat{V}$ defines a variety $V \subset \mathbb{P}({\mathcal H}_{(d)})\times \mathbb{P}(\mathbb{C}^{n+1})$ where $\mathbb{P}({\mathcal H}_{(d)})$ and
$\mathbb{P}(\mathbb{C}^{n+1})$ are the projective spaces corresponding to ${\mathcal H}_{(d)}$ and $\mathbb{C}^{n+1}$ respectively; $\widehat{V}$ and $V$ are smooth. 

Most quantities we consider are defined on $\left(\CH_{(d)} \setminus \left\{ 0 \right\}\right) \times \left(\C^{n+1} \setminus \left\{ 0 \right\}\right)$ but are constant on equivalence classes
$$ \left\{(\la_1 f, \la_2 x) \ : \ \lambda_1,\ \lambda_2 \in \mathbb{C} \setminus \{0\}\right\}$$
so are defined on $\mathbb{P}({\mathcal H}_{(d)})\times \mathbb{P}(\mathbb{C}^{n+1}).$ This product of projective spaces is the natural geometric frame for this study, but our data structure is given by pairs $(f,x) \in \left(\CH_{(d)} \setminus \left\{ 0 \right\}\right) \times \left(\C^{n+1} \setminus \left\{ 0 \right\}\right)$. We speak interchangeably of a pair $(f, \zeta)$ in $\widehat{V}$ and its projection $(f, \zeta)$ in $V$. 

Given a $C^1$ path of systems $t \in [a,b] \ra f_t \in \mathbb{P}({\mathcal H}_{(d)})$, and  a zero $\zeta_a \mathbb{P}(\mathbb{C}^{n+1})$ of the initial system $f_a$, under very general conditions, the path $t \rightarrow f_t$ can be lifted to a unique $C^1$ path $t\rightarrow (f_t,\zeta_t)$ in the solution variety $V$. 

\medskip

Two important ingredients used in this paper are {\em projective Newton's method} introduced by Shub in \cite{shu-1}, and the concept of {\em approximate zero.}

For a pair $(f, x) \in \left(\CH_{(d)} \setminus \left\{ 0 \right\}\right) \times \left(\C^{n+1} \setminus \left\{ 0 \right\}\right)$ the projective Newton's operator $N_f$ is defined by
$$N_f(x) = x - Df(x)|_{x^\perp}^{-1}f(x).$$
Here we assume that the restriction of the derivative $Df(x)$ to the subspace orthogonal to $x$
$$x^\perp = \left\{u \in \C^{n+1} \ : \ \left\langle u,x \right\rangle = 0\right\}$$
is invertible. It is easy to see that the line throught $x$ is sent by $N_f$ onto the line throught $N_f(x)$ so that $N_f$ is in fact defined on $\mathbb{P}(\mathbb{C}^{n+1}).$

\begin{definition} \label{def-approx-root} We say that $x$ is an approximate zero of $f$ with associated zero $\ze$ ($f(\ze)=0$) provided that the point
$x_{p} = N_f(x_{p-1})$, $x_0=x$, is defined for all $p \geq 1$ and
$$d_T(\ze, x_p) \leq \left(\frac{1}{2}\right)^{2^p - 1}d_T(\ze, x).$$
\end{definition}
Here $d_T$ denotes the {\it tangential ``distance''} as in Definition \ref{definitions}.
\medskip

A well-studied class of numerical algorithms for solving polynomial systems uses homotopy (or path-following) algorithms associated with a predictor-corrector scheme.
We are given a $C^1$ path $(f_t)$, ${a \leq t \leq b}$, in the space $\Hd$, and a root $\zeta_a$ of $f_a$. Under certain genericity conditions, the path $(f_t)$ may be lifted uniquely to a $C^1$ path $(f_t,\zeta_t) \in \widehat{V}$, $t \in [a,b]$, starting at the given pair $(f_a,\zeta_a)$.

Given an approximate zero $x_a$ of $f_a$ associated with $\zeta_a$, our aim is to build an approximation of this path by a sequence of pairs $(f_{t_i}, x_{i}) \in \left(\CH_{(d)} \setminus \left\{ 0 \right\}\right) \times \left(\C^{n+1} \setminus \left\{ 0 \right\}\right)$, $1 \le i \le k$ where, $a= t_0 < t_1 < \ldots < t_k = b$ is a subdivision of the interval $[a,b]$, and where $x_{i}$ is an approximate zero of $f_{t_i}$ associated with $\ze_{t_i}$. To simplify our notations we let $f_{t_i} = f_i$.

The construction of the suddivision $(t_i)$ is given in the ``Algorithm Homotopy''. The construction of the sequence $(x_i)$ uses the predictor-corrector scheme based on projective Newton's method, studied for the first time in \cite{Bez1}.  This sequence is defined recursively by
$$x_{i+1} = N_{f_{i+1}}(x_i).$$
In fact, to allow computation errors, we choose $x_{i+1}$ in a suitable neighborhood of $N_{f_{i+1}}(x_i)$.
Theorem \ref{main-homotopy} proves that for each $i=1 \cdots k,$  $x_i$ is an approximate zero of $f_{t_i}$ associated with $\ze_{t_i}$ and it gives an estimate for the integer $k$ in terms of the maximum degree $D$, and the condition length of the path $(f_t, \zeta_t)$, $a \leq t \leq b$.

\medskip

\begin{definition} \label{definitions}
As it is clear from the context, systems and vectors are supposed nonzero.
\begin{enumerate}
\item $\Hd$ is endowed with the unitarily invariant inner product (see \cite{BCSS})
$$\langle f, g \rangle = \sum_{i=1}^n \sum_{\left| \al \right| = d_i} \frac{\al_0! \dots \al_n!}{d_i!}f_{i,\al}\overline{g_{i,\al}}$$
where $\al = (\al_0, \ldots , \al_n)$, $\left| \al \right| = \al_0 + \dots + \al_n$,
$$f_i(x_0, \ldots ,x_n) = \sum_{\left| \al \right| = d_i} f_{i,\al} x_0^{\al_0} \ldots x_n^{\al_n},$$
and
$$g_i(x_0, \ldots ,x_n) = \sum_{\left| \al \right| = d_i} g_{i,\al} x_0^{\al_0} \ldots x_n^{\al_n}.$$

 \item $d_R(x,y)$ is the Riemannian distance in $\P_n(\C)$.
 \item $d_P(x,y) = \sin d_R(x,y) = \min_{0 \ne \lambda \in \mathbb C}
\frac{\|x-\lambda y\|}{\|x\|}$ is the projective distance.
 \item $d_T(x,y) = \tan d_R(x,y)$ is the tangential ``distance''. It
does not satisfy the triangle inequality.
 \item One has $d_P(x,y)
\leq d_R(x,y) \leq d_T(x,y)$.
 \item When $\rank Df(x) = n$, we denote by $\th_x$ the angle between $x$ and $\ker Df(x)$. $\th_x = 0$ when $f(x)=0$.
 \item $\th_{L,M}$ denotes the angle between the complex lines $L$ and $M$ so that $\th_x = \th_{x,\ker Df(x)}$.
 \item $\psi(u) = 1 - 4u + 2u²$ decreases from $1$ to $0$ on the interval $[0, (2-\sqrt{2})/2]$.
 \item The norm of a linear (resp. multi-linear) operator is always the operator norm.
 \item The condition number 
 $$\mu(f,x) = \lVert f \rVert \left\|  Df(x) |_{x^\perp}^{-1} \diag\left(\sqrt{d_i} \rVert x \lVert^{d_i-1}\right)\right\|,$$
 is also denoted $\mu_{proj}$ in \cite{Bez1}, and $\mu_{norm}$ in \cite{BCSS}.
 \item $D = \max d_i$. We suppose $D \geq 2$.
 \item $u = D^{3/2}\mu(f,x)d_R(x,y)/2$. 
 \item $v = D^{1/2}\mu(f,x) \left\| \frac{f}{\|f\|} - \frac{g}{\|g\|} \right\|$
where it is assumed that $f\ne 0$ and $g \ne 0$.
 \item $\be_0(f,x) = \lVert x \rVert^{-1} \lVert Df(x)|_{x^\perp}^{-1}f(x) \rVert$.
 \item $\ga_0(f,x) = \lVert x \rVert \max_{k \geq 2}  \lVert Df(x)|_{x^\perp}^{-1}\frac{D^{k}f(x)}{k!} \rVert^{1/(k-1)}$.
 \item $\al_0(f,x) = \be_0(f,x) \ga_0(f,x).$
 \item $\alpha_0 = (13-3\sqrt{17})/4 = .15767\ldots$
 \item $\de(f,x) = \lVert x \rVert^{-1} \lVert Df(x)|_{x^\perp}^{-1}\diag(d_i)f(x) \rVert$.
 \item $\phi_{t,s}(x) = \lVert x \rVert^{-1} \lVert Df_t(x)|_{x^\perp}^{-1}(f_t(x) - f_s(x)) \rVert$.
 \item Let a $C^1$ path $a \le t \le b \ra f_t \in \Hd \setminus \left\{0\right\}$ be given. We denote by
 $$\dot f_t = \frac{df_t}{dt}$$ 
 the derivative of the path with respect to $t$, and, for any $g \in \Hd$,
 $$\left\| g \right\|_{f_t} = \left\| \Pi_{f_t^\perp}g \right\|/ \left\|f_t\right\|$$
 the norm of the projection of $g$ onto the subspace orthogonal to $f_t$ divided by the norm of $f_t$.
 The length of $(f_t)$ in $\P(\Hd)$ is given by
 $$l(b) = \int_a^b {\left\| \dot f_t \right\|_{f_t}}dt.$$
 When $\left\|f_t\right\|=1$ for each $t$ we have
 $$l(b) = \int_a^b \left\| \dot f_t \right\|dt.$$
 \item The condition length of the path $(f_t, x_t) \in \left(\mathcal {H}_{(d)} \setminus \left\{0\right\} \right) \times \left(\mathbb{C}^{n+1} \setminus \left\{0\right\} \right)$, $a \le t  \le b$, is
 $$L(b) = \int_a^b \left(\left\| \dot f_t \right\|_{f_t}^2 + \left\| \dot x_t \right\|_{x_t}^2\right)^{1/2}\mu(f_t,x_t)dt,$$
where $\dot x_t$ is the derivative of the path $x_t$ with respect to $t$, and where the norm $\left\| \dot x_t \right\|_{x_t}$ is defined as in the previous item with $\left\| \dot f_t \right\|_{f_t}$. When $\left\|f_t\right\|=\left\|x_t\right\|=1$ for each $t$ we have
$$L(b) = \int_a^b \left(\left\| \dot f_t \right\|^2 + \left\| \dot x_t \right\|^2\right)^{1/2}\mu(f_t,x_t)dt.$$

We will also use some invariants related with non homogeneous polynomial systems. For $(d) = (d_1,\ldots,d_n) \in {\mathbb N}^n$, let ${\mathcal P}_{(d)} = \prod_{i=1}^n {\mathcal P}_{d_i}$ be the set of polynomial systems $F = (F_1, \ldots , F_n) : \mathbb{C}^{n}\rightarrow \mathbb{C}^n$ in the variables $X = (X_1, \ldots , X_n)$, of respective degrees $\deg(F_i) \leq d_i,~1\leq i\leq n$. The homogeneous counterpart of $F$ is the system $f=(f_1,\ldots,f_n) \in \Hd$ defined by
$$f_i(x_0, x_1, \ldots ,x_n) = x_0^{d_i}F_i\left(\frac{x_1}{x_0}, \ldots ,\frac{x_n}{x_0}\right).$$ The norm on ${\mathcal P}_{(d)}$ is defined by $\left\|F\right\| = \left\|f\right\|$. We also let:
\item $\be(F,X) =  \lVert DF(X)^{-1}F(X) \rVert$.
\item $\ga(F,X) =  \max_{k \geq 2}  \lVert DF(X)^{-1}\frac{D^{k}F(X)}{k!} \rVert^{1/(k-1)}$.
\item $\al(F,X) = \be(F,X) \ga(F,X).$
\end{enumerate}
\end{definition}

\section{Variation of the condition number}\label{sect-3}

A necessary ingredient for the proof of Theorem~\ref{main-homotopy} is the following theorem which gives the variations of the condition number when both the system and the point vary:

\begin{theorem} \label{the-1}
Let two nonzero systems $f, g \in {\mathcal H}_{(d)}$, and two nonzero vectors $x, y \in \C^{n+1}$ be given such that
$\rank  Df(x) |_{x^\perp} = n$, $u \leq 1/20$, and $v \leq 1/20$. Then $\rank  Dg(y) |_{y^\perp} = n$, and
$$(1 - 3.805 u - v) \mu(g,y) \leq \mu(f,x) \leq (1 + 3.504 u + v) \mu(g,y).$$
\end{theorem}

\begin{corollary} \label{cor-1}  Let $0 < \ep \leq 1/4$, two nonzero systems $f, g \in {\mathcal H}_{(d)}$,
and two nonzero vectors $x, y \in \C^{n+1}$ be given such that $u \leq \ep/5$, $v \leq \ep/5$, and
$\rank  Df(x) |_{x^\perp} = n$. One has $\rank  Dg(y) |_{y^\perp} = n$, and
$$(1 - \ep) \mu(g,y) \leq \mu(f,x) \leq (1 + \ep) \mu(g,y).$$
\end{corollary}

The proof of these results is obtained from the following series of lemmas.

\begin{lemma}\label{high-der}
Let $f \in \Hd$ and $x \in \C^{n+1}$. We have
\begin{enumerate}
\item For any $i=1 \ldots n$,  $\left\|f_i(x)\right\| \leq \left\| f_i \right\| \left\| x \right\|^{d_i}$, so that $\left\|f(x)\right\| \leq 1$ when $f$ and $x$ are normalized.
\item For any $i=1 \ldots n$,  $\left\|Df_i(x)\right\| \leq d_i \left\| f_i \right\| \left\| x \right\|^{d_i - 1}$, so that $\left\| Df(x) \right\| \leq D$ when $f$ and $x$ are normalized.
\item $\gamma_0(f,x) \le \frac{D^{3/2}}{2} \mu(f,x).$
\end{enumerate}
\end{lemma}

\proof These inequalities come from Proposition 1, and Theorem~2, p. 267, in \cite{BCSS}.
\endproof

\begin{lemma} \label{lem-1} $1 \leq \sqrt{n} = \min_{f,x}\mu(f,x).$
\end{lemma}

\proof Let given a matrix $A \in \C^{n \times m}$, $m \geq n$, and $A = U \Si V^*$ a singular value decomposition with
$$\Si = \diag(\si_1(A) \geq \ldots \geq \si_n(A))\in \C^{n \times m},$$
$\si_i(A) \geq 0$, $U \in \U_n$, $V \in \U_m$ unitary matrices.
We define
$$\ka(A) = \si_n(A)^{-1} \lVert A \rVert_F  = \si_n(A)^{-1} \sqrt{\si_1(A)² + \ldots + \si_n(A)²}$$
when $\si_n(A) > 0$, and $\infty$ otherwise. We see easily that
$$\min_A \ka(A) =  \sqrt{n},$$
and this minimum is obtained when $\si_i(A) = 1$, $1 \leq i \leq n$.

For the case of polynomial systems we have
$$\mu(f,x) = \lVert f \rVert \lVert \left(  \diag(d_i^{-1/2}) Df(x) |_{x^\perp}\right)^{-1} \rVert =
\lVert f \rVert \si_n\left( \diag(d_i^{-1/2}) Df(x) |_{x^\perp}\right) ^{-1}.$$
Using the unitary invariance of the norm in ${\mathcal H}_{(d)}$, that is $\lVert f \circ U \rVert = \lVert f \rVert$ for any $U \in \U_{n+1}$,
considering a $U$ such that $Ue_0 = x$ with $e_0 = (1, 0 , \ldots , 0)^T \in \C^{n+1},$ we see that
$$\lVert   \diag(d_i^{-1/2}) Df(x) |_{x^\perp} \rVert_F \leq \lVert f \rVert.$$
Thus, our previous estimate shows that $\mu(f,x) \geq \sqrt{n}.$

To prove the equality $\sqrt{n} = \min_{f,x}\mu(f,x)$ we use the unitary invariance of the condition number
$$\mu(f,x) = \mu(f \circ U, U^* x)$$
for any $U \in \U_{n+1}$, and the equality $\sqrt{n} = \mu(f,e_0)$ when $f_i(z) = \sqrt{d_i}z_0^{d_i - 1}z_i,$ $1 \leq i \leq n$,
and $e_0 = (1, 0 , \ldots , 0)^T \in \C^{n+1}.$
\endproof

\begin{lemma}  \label{lem-0} For any $f$ and $x$ one has $d_R(x,y) \leq u/ \sqrt{2n} \leq u/ 2$, and $d_R(x,y)\de(f,x) \le u$.
\end{lemma}

\proof We suppose that both system $f$ and vectors $x, y$ are normalized. We have by Lemma \ref{lem-1}
$$u = \frac{D^{3/2}}{2}\mu(f,x)d_R(x,y) \geq \frac{2^{3/2}}{2}\sqrt{n}d_R(x,y).$$
For the second inequality we have
\begin{eqnarray*}
d_R(x,y)\de(f,x) &\leq& d_R(x,y)\lVert Df(x)|_{x^\perp}^{-1}\diag(d_i^{1/2}) \rVert
\lVert \diag(d_i^{1/2}) \rVert \lVert f(x) \rVert \\
&\leq&
d_R(x,y) \mu(f,x) D^{1/2} \lVert f(x) \rVert \\
&\leq& d_R(x,y) \mu(f,x) D^{3/2}/2 \\ &=& u
\end{eqnarray*}
thanks to the inequalities $2 \le D$, $\lVert f_i(x) \rVert \leq \lVert f_i \rVert \lVert x \rVert^{d_i}$
(Lemma \ref{high-der}), and the hypothesis $\lVert f \rVert = \lVert x \rVert = 1$ .
\endproof

\begin{lemma}  \label{lem-2} When $\rank  Df(x) |_{x^\perp} = n$ we have
$$\rVert Df(x)|_{x^\perp}^{-1}Df(x)|_{y^\perp} \lVert \leq 1 + d_P(x,y) \tan \th_x.$$
\end{lemma}

\proof
%
%
%
%
%
Take $u \in y^\perp$ and define $v =  Df(x)|_{x^\perp}^{-1}Df(x)|_{y^\perp} u$ so that
$$v = \left( u + \ker Df(x) \right) \cap x^\perp.$$
$v \in x^\perp$ is the projection of $u \in y^\perp$ along $\ker Df(x)$.
Let us denote by $w$ the orthogonal projection of $u$ onto $x^\perp$. See Figure 1.
\begin{figure}[htbp]
\centering
\includegraphics[angle = 90, width=0.8 \textwidth]{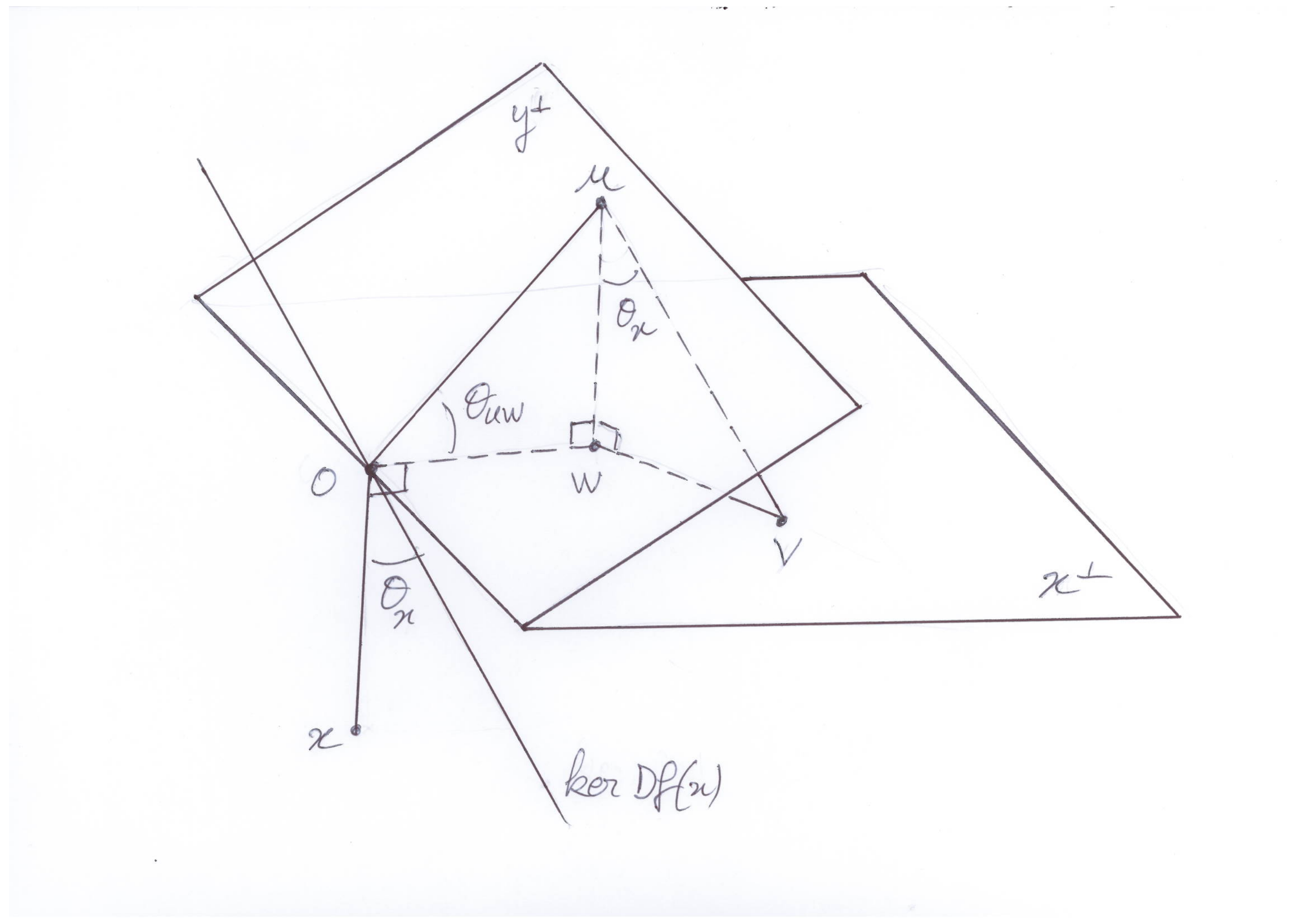}
\caption{}
\end{figure}
We have $\| w \| = \| u \| \cos \th_{u,w}$  and
$$\| w-v \| = \| w-u \| \tan \th_{x} =  \| u \| \sin \th_{u,w} \tan \th_{x},$$
so that
$$\| v \| \leq \| u \| (\cos \th_{u,w} + \sin \th_{u,w} \tan \th_{x}) \le  \| u \| (1 + \sin \th_{x,y} \tan \th_{x})$$
because $\th_{x^\perp,y^\perp} = \th_{x,y} = \max_{\begin{array}[pos]{l}
	\| u \| = 1\\
        u \in y^\perp
\end{array}
} \th_{u,w}.$ \endproof

\begin{lemma}  \label{lem-3} Assume that $\|x\|=\|y\|=1$. 
When $\rank  Df(x) |_{x^\perp} = n$ and $u<1$ one has
$$\rVert Df(x)|_{x^\perp}^{-1}Df(y)|_{y^\perp} \lVert \leq \frac{1}{(1-u)²} + d_P(x,y) \tan \th_x.$$
\end{lemma}

\proof Since $Df(y) = Df(x) + \sum_{k \geq 1} \frac{D^{k+1}f(x)}{k!}(y-x)^k$ we get
$$ Df(x)|_{x^\perp}^{-1} Df(y)|_{y^\perp} =  Df(x)|_{x^\perp}^{-1} Df(x)|_{y^\perp} +
\sum_{k \geq 1} (k+1)  Df(x)|_{x^\perp}^{-1}\frac{D^{k+1}f(x)}{(k+1)!}(y-x)^k|_{y^\perp}.$$
Then, we use $\lVert Df(x)|_{x^\perp}^{-1}\frac{D^{k+1}f(x)}{(k+1)!} \rVert \leq \ga_0(f,x)^{k}$, Lemma \ref{lem-2} and
Lemma \ref{high-der} to obtain
\begin{eqnarray*}
\rVert Df(x)|_{x^\perp}^{-1}Df(y)|_{y^\perp} \lVert &\leq& 1 + d_P(x,y) \tan \th_x +
\sum_{k \geq 1} (k+1) \ga_0(f,x)^{k}\lVert x-y \rVert^k \\
&\leq&
\frac{1}{(1-u)²} + d_P(x,y) \tan \th_x
\end{eqnarray*}
\endproof

%

\begin{lemma}  \label{lem-5}  When $\rank  Df(x) |_{x^\perp} = n$ then
$$\tan \th_x = \de(f,x) \leq D^{1/2} \mu(f,x).$$
\end{lemma}

\proof We assume that $x$ and $f$ are normalized. Let
$y = Df(x)|_{x^\perp}^ {-1}Df(x)x$ be the projection of $x$ onto $x^ \perp$ along $\ker Df(x)$ so that
$\Vert y \Vert = \tan \th_{x}.$
By Euler's identity, one also has
$$y =  Df(x)|_{x^\perp}^ {-1}\diag(d_i)f(x)$$
thus,
\begin{eqnarray*}
\tan \th_x
&=& \de(f,x) \\
&\leq&
 \rVert Df(x)|_{x^\perp}^ {-1}\diag(d_i^{1/2})\lVert \rVert \diag(d_i^{1/2}) f(x)\lVert
\end{eqnarray*}
Using Lemma \ref{high-der} we obtain
$$\tan \th_x \leq D^{1/2}  \mu(f,x) .$$
\endproof

\begin{lemma}  \label{lem-6}  When $\rank  Df(x) |_{x^\perp} = n$ one has
$$\rVert Df(x)|_{x^\perp}^{-1}Df(x)|_{y^\perp} \lVert \leq 1 + \de(f,x) d_P(x,y).$$
Moreover, when $\|x\|=\|y\|$ and $u<1$, we have
$$\rVert Df(x)|_{x^\perp}^{-1}Df(y)|_{y^\perp} \lVert \leq \frac{1}{(1-u)²} + \de(f,x) d_P(x,y).$$
\end{lemma}

\proof The first assertion comes from lemmas \ref{lem-2} and \ref{lem-5}. The second assertion is a consequence of
lemmas \ref{lem-3} and \ref{lem-5}.
\endproof

\begin{lemma} \label{lem-7} Assume that $\|x\|=\|y\|=1$.
When $\rank  Df(x) |_{x^\perp} = n$, and $u < (2 - \sqrt{2})/2$,
then $\rank Df(y)|_{x^\perp} = n$ and
$$\lVert Df(y)|_{x^\perp}^{-1}Df(x)|_{x^\perp} \rVert \leq \frac{(1-u)²}{\psi(u)}.$$
Moreover, if $u \le 1/19$,
\[
\lVert Df(y)|_{y^\perp}^{-1} Df(x)|_{x^\perp} \rVert
\leq
(1+d_P(x,y) \delta(f,y)) \frac{(1-u)^2}{\psi(u)}
\leq
1+3.805 u .
\]
\end{lemma}

\proof We have
$$Df(y) = Df(x) + \sum_{k \geq 1} \frac{D^{k+1}f(x)}{k!}(y-x)^k$$
so that
$$Df(x)|_{x^\perp}^{-1}(Df(y) - Df(x)) =   \sum_{k \geq 1} Df(x)|_{x^\perp}^{-1}\frac{D^{k+1}f(x)}{k!}(y-x)^k,$$
and, like in the proof of Lemma \ref{lem-3},
$$\lVert Df(x)|_{x^\perp}^{-1}(Df(y) - Df(x)) \rVert \leq  \sum_{k \geq 1} (k+1)u^k = \frac{2u-u^2}{(1-u)^2} < 1$$
(because $u < (2 - \sqrt{2})/2$) so that, by Neumann's Perturbation Theorem,
$$I_{x^\perp} + Df(x)|_{x^\perp}^{-1}(Df(y) - Df(x))|_{x^\perp} = Df(x)|_{x^\perp}^{-1}Df(y)|_{x^\perp}$$
is invertible (i.e. $Df(y)|_{x^\perp}$ is invertible), and
$$\lVert Df(y)|_{x^\perp}^{-1}Df(x)|_{x^\perp} \rVert \leq \frac{1}{1 -  \frac{2u-u^2}{(1-u)^2}} = \frac{(1-u)²}{\psi(u)}.$$
\medskip
\par

We will prove the second statement in two steps.
First, prove it under the assumption that $Df(y)|_{y^\perp}$ is
invertible. Then, we remove this assumption.

The first step goes as follows.
Combining the first statement
with Lemma~\ref{lem-2} we obtain:

\begin{eqnarray*}
\| Df(y)|_{y^\perp}^{-1} Df(x)|_{x^{\perp}} \|
&\le&
\| Df(y)|_{y^\perp}^{-1} Df(y)|_{x^{\perp}} \|
\| Df(y)|_{x^\perp}^{-1} Df(x)|_{x^{\perp}} \|
\\
&\le&
(1+d_P(x,y) \tan \theta_y) \frac{(1-u)^2}{\psi(u)}
\\
&=&
(1+d_P(x,y) \delta(f,y)) \frac{(1-u)^2}{\psi(u)}
\end{eqnarray*}

A bound for $\delta(f,y)$ is
\begin{eqnarray*}
\delta(f,y) &=&
\left\| Df(y)_{y^\perp}^{-1} \diag(d_i) f(y) \right\| \\
&\le&
\| Df(y)|_{y^\perp}^{-1} Df(x)|_{x^{\perp}} \|
\| \left\| Df(x)_{x^\perp}^{-1} \diag(d_i) f(y) \right\| \\
&\le&
\| Df(y)|_{y^\perp}^{-1} Df(x)|_{x^{\perp}} \| \mu(f,x) \sqrt{D}
\end{eqnarray*}
using $|f(y)| \le 1$ since $\|f\|=1$ and $\|y\|=1$.

Combining both inequations and setting $M=
\| Df(y)|_{y^\perp}^{-1} Df(x)|_{x^{\perp}} \|$, we obtain:
\[
M \le (1 + M u) \frac{(1-u)^2}{\psi(u)}
\]
that simplifies to
\[
M \le \frac{ (1-u)^2 } {\psi(u) - u(1-u)^2}=
\frac{(1-u)^2}{1 - 5u + 4 u^2 - u^3 }
\le
1+3.805 u
.
\]
The last bound follows from the fact that the numerator
and the denominator have alternating signs, so the Taylor
expansion at zero of the fraction has terms of the same
sign (positive).
Hence,
\[
\frac{ \frac{(1-u)^2}{1 - 5u + 4 u^2 - u^3 }-1}{u}
\]

is an increasing function. In particular, for $u=1/19$, this
is smaller than $3.805$.
\medskip
\par

Now, we must prove that $Df(y)|_{y^\perp}$ is invertible.
Let $(x_t)_{t \in [0, d_R(x,y)]}$ denote
a minimizing geodesic (arc of great circle) between $x$ and $y$.

Let $W$ be the subset of all $t \in  [0, d_R(x,y)]$ so that
$Df(x_t)|_{x_t}$ is invertible. It is an open set, and
$0 \in W$.

We claim that $W$ is a closed set. Indeed, let $s \in \overline W$.
Then there is a sequence of $t_i \in W$ with $t_i \rightarrow s$.
We know from the second statement (restricted) that
\[
\|
Df(x_{t_i})|_{x_{t_i}^{\perp}} ^{-1}
Df(x)_{x^{\perp}}
\| \le 1 + 3.805 u.
\]

Hence, for $\tau = t_i$,
\[
h(\tau) = \|
Df(x_{\tau})|_{x_{\tau}^{\perp}} ^{-1}
Df(x)_{x^{\perp}}
\|_F^2 \le n (1 + 3.805 u)^2
\]

The function $h(\tau)$ is a rational function of a real parameter $\tau$,
so its domain is an open set and contains $s$. By continuity,
$h(s) \le n(1+3.805 u)^2$. Thus, $Df(x_s)_{x_s^{\perp}}$ is invertible,
and $s \in W$. As $W$ is a non-empty  open and closed subset of an
interval, $W = [0, d_R(x,y)]$ and $Df(y)|_{y^{\perp}}$ must be invertible.
\endproof

\begin{lemma}  \label{lem-8}  Suppose that $\rank  Df(x) |_{x^\perp} = n$, $u < 1$ and $\mu(f,y)$ is finite. Then
$$\mu(f,x) \leq \mu(f,y) \left(\frac{1}{(1-u)²} + \de(f,x) d_P(x,y)\right).$$
\end{lemma}

\proof Suppose that $\rVert x \lVert = \rVert y \lVert = \rVert f \lVert = 1.$
We can bound
$$\mu(f,x) = \lVert  Df(x) |_{x^\perp}^{-1}  Df(y) |_{y^\perp} Df(y) |_{y^\perp}^{-1} \diag(\sqrt{d_i})\rVert$$
and we conclude with Lemma \ref{lem-6}.
\endproof

\begin{lemma}  \label{lem-9}  When $\rank  Df(x) |_{x^\perp} = n$, and
$u<1/19$
we have
$$\mu(f,y) \leq
(1+3.805u) \mu(f,x).$$
\end{lemma}

\proof Suppose that $\rVert x \lVert = \rVert y \lVert = \rVert f \lVert = 1.$ We have
$$\mu(f,y) = \lVert  Df(y)|_{y^\perp}^{-1} Df(x)|_{x^\perp} Df(x)|_{x^\perp}^{-1} \diag(\sqrt{d_i})\rVert \leq (1+3.805u)
\mu(f,x)$$
by Lemma \ref{lem-7}.
\endproof

\begin{lemma}  \label{lem-10}  Assume that $\|f\|=\|g\|=1$. 
Suppose that $\rank  Df(x) |_{x^\perp} = n$, and $v<1$. Then $\rank  Dg(x) |_{x^\perp} = n$, and
$$\lVert Df(x) |_{x^\perp}^{-1}Dg(x) |_{x^\perp} \rVert \leq 1+v,$$
$$\lVert Dg(x) |_{x^\perp}^{-1}Df(x) |_{x^\perp} \rVert \leq \frac{1}{1-v},$$
$$(1-v)\mu(g,x) \leq \mu(f,x) \leq (1+v)\mu(g,x).$$
\end{lemma}

\proof  Suppose that $\rVert x \lVert = 1.$ One has
$Df(x) |_{x^\perp}^{-1}Dg(x) |_{x^\perp} = I_{x^\perp} - (I_{x^\perp} - \mbox{ idem})$ and
$$I_{x^\perp} - \mbox{ idem} = Df(x) |_{x^\perp}^{-1} \diag(d_i^{1/2})\diag(d_i^{-1/2})D(f-g)(x) |_{x^\perp}$$
which norm is bounded by (using Lemma \ref{high-der})
$$\mu(f,x)D^{1/2}\Vert f-g \Vert \le v < 1.$$
This proves the first inequality.
Thus $Df(x) |_{x^\perp}^{-1}Dg(x) |_{x^\perp}$ is invertible and the norm of its inverse is bounded by $1/(1-v)$ (Neumann's Perturbation Theorem). This gives
$$ \mu(g,x) \leq \lVert Dg(x) |_{x^\perp}^{-1}Df(x) |_{x^\perp} \rVert \lVert Df(x) |_{x^\perp}^{-1} \diag((d_i^{1/2}) \rVert \leq \frac{\mu(f,x)}{1-v}.$$
The last inequality is obtained via
$$\mu(f,x) \leq \lVert Df(x) |_{x^\perp}^{-1}Dg(x) |_{x^\perp} \rVert \lVert Dg(x) |_{x^\perp}^{-1} \diag((d_i^{1/2}) \rVert
\leq (1+v)\mu(g,x).$$
\endproof

\begin{lemma}  \label{lem-11} Let $u_g = D^{3/2}\mu(g,x)d_R(x,y)/2$. Suppose that $\rank  Df(x) |_{x^\perp} = n$, and $v<1$. Then
$$(1-v)u_g \leq u \leq (1+v)u_g,$$
and
$$(1-v)\de(g,x) - v \leq \de(f,x) \leq (1+v)\de(g,x) + v.$$
\end{lemma}

\proof  Suppose that $\rVert x \lVert = 1.$ The first double inequality is a consequence of Lemma \ref{lem-10}.
For the second one, one has:
$$\de(g,x) = \lVert Dg(x)|_{x^\perp}^{-1} Df(x)|_{x^\perp} Df(x)|_{x^\perp}^{-1} \diag(d_i) f(x) +
Dg(x)|_{x^\perp}^{-1}  \diag(d_i) (g(x) - f(x)) \rVert.$$
By Lemma \ref{lem-10}$\lVert Dg(x)|_{x^\perp}^{-1} Df(x)|_{x^\perp} \rVert \leq 1/(1-v)$ so that
$$\de(g,x) \leq \frac{\de(f,x)}{1-v} + \mu(g,x) D^{1/2} \Vert f-g \Vert .$$
Again by Lemma \ref{lem-10} $\mu(g,x) \leq \mu(f,x)/(1-v)$ so that
$$\de(g,x) \leq \frac{\de(f,x)}{1-v} + \frac{\mu(f,x) D^{1/2} \Vert f-g \Vert}{1-v} \le \frac{\de(f,x) + v}{1-v}.$$
Similarly $\de(f,x) = $
$$\lVert Df(x)|_{x^\perp}^{-1} Dg(x)|_{x^\perp} Dg(x)|_{x^\perp}^{-1} \diag(d_i) g(x) +
Df(x)|_{x^\perp}^{-1}  \diag(d_i) (f(x) - g(x)) \rVert \leq$$
$$\lVert Df(x)|_{x^\perp}^{-1} Dg(x)|_{x^\perp} \rVert \de(g,x) + v.$$
Lemma \ref{lem-10} shows that $\lVert Df(x)|_{x^\perp}^{-1} Dg(x)|_{x^\perp} \rVert \leq 1+v$ so that
$$\de(f,x) \leq (1+v) \de(g,x) + v.$$
\endproof

\begin{lemma}  \label{lem-12}  When $\rank  Df(x) |_{x^\perp} = n$, $u, v \le 1/20$, then
$$\mu(g,y)
\frac{(1-v)}{1 + 3.805
\frac{u}{1-v} } \leq \mu(f,x) \leq
(1+v)\left( \frac{1}{\left(1-\frac{u}{1-v}\right)²} + \frac{u}{1-v}\right)\mu(g,y).$$
\end{lemma}

\proof
As before, let $u_g = D^{3/2} \mu(g,x) d_R(x,y)/2$.
Lemma~\ref{lem-11} allows us to bound $u_g \le u/(1-v) \le 1/19$.
So we may apply Lemma \ref{lem-9} to $g$ instead of $f$ so that
$$\mu(g,y) \leq (1+3.805 u_g) \mu(g,x)$$
Then, we bound $\mu(g,x)$ by $\mu(f,x)/(1-v)$ (Lemma \ref{lem-10}) and
obtain the first inequality. In particular, $\mu(g,y)$ is finite.

To prove the second one we apply Lemma \ref{lem-10}, and Lemma \ref{lem-8} to obtain
$$\mu(f,x) \leq (1+v) \mu(g,x) \leq (1+v)\left(\frac{1}{\left(1-u_g\right)²} + \de(g,x)d_P(x,y)\right)\mu(g,y) .$$
By Lemma \ref{lem-0} and Lemma \ref{lem-11} we have
$$ \de(g,x)d_P(x,y) \le u_g \le \frac{u}{1-v}$$
and we are done.
\endproof

\subsection{Proof of Theorem \ref{the-1}}

To prove the inequalities $(1 - 3.805 - v)\mu(g,y) \leq \mu(f,x) \leq (1+3.504u+v) \mu(g,y)$ we use Lemma \ref{lem-12} which gives
$$B(u,v)\mu(g,y) \leq \mu(f,x) \leq A(u,v) \mu(g,y)$$
with
$$A(u,v) = (1+v)\left( \frac{1}{\left(1-\frac{u}{1-v}\right)²} + \frac{u}{1-v}\right), \
B(u,v) = \frac{(1-v)}{1 + 3.805 \frac{u}{1-v}
} .$$

\[
A(u,v) = 1 + 3u + v + u
\frac
{6v^3+9uv^2-12v^2+4u^2v-12uv+6v-2u^2+3u}
{\left(1-v\right) \left(1-u-v\right)^2}
\]
and the last parenthesis is less than $0.504$ when $u,v < 1/20$.

The function $B(u,v) - (1-3.805 u - v)$ is increasing in $u$ and $v$,
and vanishes at the origin.
\endproof

\subsection{Proof of Corollary \ref{cor-1}}

Since $u,\ v \leq \ep / 5$ and $\ep \leq 1/4$ we get $u,v \leq 1/20$. Thus we can apply Theorem \ref{the-1} which gives
$$(1 - \ep)\mu(g,y) \leq (1 - 3.805u - v)\mu(g,y) \leq \mu(f,x) \leq (1+3.504u+v) \mu(g,y) \leq (1+\ep) \mu(g,y).$$
\endproof 
\remove{

\subsection{Estimates for $\beta_0$}

\begin{lemma}  \label{lem-13} Suppose that $\rank  Df(x) |_{x^\perp} = n$, and $v<1$. Then $\rank  Dg(x) |_{x^\perp} = n$, and
$$(1-v)\be_0(g,x) - D^{-1/2}v \leq \be_0(f,x) \leq (1+v)\be_0(g,x) + D^{-1/2}v,$$
and
$$(1-v)\de(g,x) - v \leq \de(f,x) \leq (1+v)\de(g,x) + v,$$
\end{lemma}

\proof Suppose that $\left\|x\right\| = \left\|f\right\| = \left\|g\right\| = 1$.
$$\be_0(g,x) = \lVert Dg(x)|_{x^\perp}^{-1} Df(x)|_{x^\perp} Df(x)|_{x^\perp}^{-1} \left( f(x)
+ (g(x) - f(x))\right)  \rVert \leq $$
$$\lVert Dg(x)|_{x^\perp}^{-1} Df(x)|_{x^\perp} \rVert \left( \lVert Df(x)|_{x^\perp}^{-1} f(x) \rVert
+ \lVert Df(x)|_{x^\perp}^{-1}(g(x) - f(x)) \rVert \right). $$
Since
$$\lVert Df(x)|_{x^\perp}^{-1}(g(x) - f(x)) \rVert \leq \mu(f,x) \lVert \diag(d_i^{-1/2})(g(x) - f(x)) \rVert \leq $$
$$\mu(f,x) \lVert g - f \rVert = D^{-1/2}v,$$
by Lemma \ref{lem-10} we obtain
$$\be_0(g,x)\leq \frac{1}{1-v}\left( \be_0(f,x) + D^{-1/2}v \right).$$

The other inequality is proved similarly:
$$\be_0(f,x) = \lVert Df(x)|_{x^\perp}^{-1} Dg(x)|_{x^\perp} Dg(x)|_{x^\perp}^{-1} g(x)
+ Df(x)|_{x^\perp}^{-1}(g(x) - f(x))  \rVert .$$
By Lemma \ref{lem-10} we obtain
$$\be_0(f,x) \leq (1+v)\be_0(g,x) + D^{-1/2}v.$$

The inequalities for $\de$ are proved by a similar argument.
\endproof

\begin{lemma}  \label{lem-14} Suppose that $\rank  Df(x) |_{x^\perp} = n \ge 2$, $\left\| x \right\| = \left\| y \right\|$ and $u<1$. Then
$$\left\| x \right\|^{-1}\lVert Df(x)|_{x^\perp}^{-1} \left( f(x) -f(y) \right)  \rVert \leq
\frac{\lVert x-y  \rVert}{\left\| x \right\|} \left( \frac{1}{1-u} + \de(f,x) + 1\right) \leq \frac{u/2}{1-u} + u. $$
\end{lemma}

\proof Suppose that $\left\|x\right\| = 1$. By Taylor's formula
\begin{eqnarray*}
Df(x)|_{x^\perp}^{-1} \left( f(y) -f(x) \right) &=&
Df(x)|_{x^\perp}^{-1}  Df(x) (y-x)
+ \sum_{k \geq 2}  Df(x)|_{x^\perp}^{-1} \frac{D^kf(x)}{k!}(y-x)^k
\\
&=& z + \sum_{k \geq 2} Df(x)|_{x^\perp}^{-1} \frac{D^kf(x)}{k!}(y-x)^k
\end{eqnarray*}
which norm is bounded by
$$\lVert z \rVert + \sum_{k \geq 2}\ga_0(f,x)^{k-1}\lVert x-y  \rVert^k \leq \lVert z \rVert + \lVert x-y  \rVert
\frac{\ga_0(f,x) \lVert x-y  \rVert}{1 - \ga_0(f,x) \lVert x-y  \rVert} \leq$$
$$\lVert z \rVert + \lVert x-y  \rVert \frac{u}{1-u}.$$
The vector $z$ is the projection of $y-x$ on $x^\perp$ along $\ker Df(x)$. Let $w$ denotes the orthogonal projection of $y-x$ on $x^\perp$ so that $w = y - \left\langle y,x \right\rangle x$. Since the angle at $y-x$ of the triangle with vertices at $z$, $y-x$, and $w$ is equal to $\th_x$, we have
\begin{figure}[htbp]
\centering
\includegraphics[angle = 90, width=0.8 \textwidth]{DMS-fig-3}
\caption{}
\end{figure}
$$\lVert z \rVert \leq \lVert z-w \rVert + \lVert w \rVert = \lVert y-x-w \rVert \tan \th_x + \lVert y - \left\langle y,x \right\rangle x \rVert =$$
$$\rvert \left\langle y-x,x \right\rangle \lvert \tan \th_x + \sqrt{1 - \rvert \left\langle y,x \right\rangle \lvert²}
\leq \rVert y-x \lVert ( \tan \th_x + 1).$$
According to Lemma \ref{lem-5}
$$\lVert z \rVert \leq  \rVert y-x \lVert (\de(f,x) + 1)$$
so that
$$\left\| x \right\|^{-1}\lVert Df(x)|_{x^\perp}^{-1} \left( f(x) -f(y) \right)  \rVert \leq \frac{\rVert y-x \lVert}{\rVert x \lVert}
\left( \frac{u}{1-u} + \de(f,x) + 1 \right).$$
To finish this proof, we use the inequalities $\| x-y \| / \| x \| \leq d_R(x,y) \leq u / 2$ and $d_R(x,y)\de(f,x) \leq u$ (Lemma \ref{lem-0}).
\endproof

\begin{lemma}  \label{lem-15} Suppose that $\rank  Df(x) |_{x^\perp} = n\ge 2$, $\left\|x\right\| = \left\|y\right\|$ and $u<1$. One has
$$\frac{1}{1+3.805u}
\be_0(f,y) - \frac{u/\sqrt{2}}{1-u} - u \leq
  \be_0(f,x) \leq
\left( \frac{1}{(1-u)^2} + u \right) \be_0(f,y) + \frac{u/2}{1-u} + u.$$
For the second inequality, we assume that $\be_0(f,y)$ is finite.
\end{lemma}

\proof Suppose that $\left\| x \right\| = \left\| y \right\| = 1. $
$$\be_0(f,x) = \lVert Df(x)|_{x^\perp}^{-1} Df(y)|_{y^\perp} Df(y)|_{y^\perp}^{-1}f(y) +
 Df(x)|_{x^\perp}^{-1}\left( f(x) -f(y) \right)  \rVert \leq $$
 $$\lVert Df(x)|_{x^\perp}^{-1} Df(y)|_{y^\perp}\rVert \lVert  Df(y)|_{y^\perp}^{-1}f(y)\rVert +
\lVert Df(x)|_{x^\perp}^{-1}\left( f(x) -f(y) \right)  \rVert $$
so that, from Lemma \ref{lem-0}, Lemma \ref{lem-6}, and Lemma \ref{lem-14} we get
$$\leq \left( \frac{1}{(1-u)^2} + \de(f,x)d_P(x,y)\right) \be_0(f,y) + \frac{u/2}{1-u} + u \le $$
$$\left( \frac{1}{(1-u)^2} + u \right) \be_0(f,y) + \frac{u/2}{1-u} + u .$$
Let us now prove the other inequality:
$$\be_0(f,y) \leq \lVert Df(y)|_{y^\perp}^{-1}  Df(x)|_{x^\perp}\rVert
\left( \right. \lVert Df(x)|_{x^\perp}^{-1} f(x) \rVert + \lVert Df(x)|_{x^\perp}^{-1}\left( f(y) -f(x) \right)  \rVert \left. \right). $$
From Lemma \ref{lem-7}  and Lemma \ref{lem-14} we get
$$\be_0(f,y) \leq (1+3.805u)
\left( \be_0(f,x) + \frac{u/2}{1-u} + u \right).$$
\endproof

\medskip

We can compose Lemma \ref{lem-13} and Lemma \ref{lem-15} to obtain a double inequality which relies $\be_0(f,x)$ and $\be_0(g,y)$.
}

\section{Alpha theory in projective spaces}\label{sec-alpha-th}

\begin{theorem}\label{the-3}
Let $0< \al \le \alpha_0 = (13-3\sqrt{17})/4 =  0.15767\ldots$ Let $f \in \Hd$ and $x \in \C^{n+1}$ both nonzero. If
\[
\frac{D^{3/2}}{2}\mu(f,x)\be_0(f,x) \le \al
,
\]
then there is a zero $\zeta \in \C^{n+1}$ of $f$ satisfying:
$
d_{T} (x, \zeta) \le \sigma(\al) \beta_0(f,x)
$
with
\[
\sigma(\al) =
\frac{1}{4} +
\frac{1 - \sqrt{(1 + \al)^2 - 8 \al}}{4 \al}
.
\]
Furthermore, if $y=N_f(x)$, then $d_{R} (y, \zeta) \le (\sigma(\al)-1) \beta(f,x)$.
Moreover, when
$$\frac{D^{3/2}}{2}\mu(f,x)\be_0(f,x) \leq \al \leq 0.049,$$
then $x$ is an approximate zero of $f$ corresponding to $\ze$,
and so does $y$.
\end{theorem}

\begin{proof}
The proof below follows the lines of~\cite{Bez5}. We suppose that $f$ and $x$ are normalized ($\left\|f\right\| = \left\|x\right\| = 1$).
We consider the non-homogeneous polynomial system, defined for a variable $X \in x^{\perp}$ by
\[
F(X) = f( x + X)
.
\]
Let us denote by $N_F$ the usual Newton operator: $$N_F(X) = X - DF(X)^{-1}F(X).$$
Then $DF(X) = Df(x+X)|_{x^{\perp}}$. In particular,
$DF(0) = Df(x)|_{x^{\perp}}$ and we have,
\begin{eqnarray*}
y = N_f(x) &=& \lambda ( x + N_F(0) ), \lambda \in \mathbb C \setminus \{0\} \\
\beta_0(f,x) &=& \beta(F, 0) \\
\gamma_0(f,x) &\ge& \gamma(F, 0) \\
\alpha_0(f,x) &\ge& \alpha(F,0).
\end{eqnarray*}
Since, by Lemma \ref{high-der}, $\alpha_0(f,x) \leq \be_0(f,x)\mu(f,x){D^{3/2}}/{2}$, by Theorem 1, p. 462, in \cite{Bez1}, $0$ is an approximate zero of $F$ and hence (Definition 1 ibid.) the sequence $(X_k)_{k {\ge 0}}$
defined recursively by $X_{k+1} = N_F(X_k)$, $X_0=0$, converges quadratically to a zero $Z$ of $F$. Namely,
\[
\| X_{k+1} - X_k \| \le
 2^{-2^k+1} \|X_1 - X_0\|
.
\]
Moreover, by the same theorem,
\[
\|Z - X_0\| \le \frac{1+\alpha(F,0)-\sqrt{(1+\alpha(F,0))^2-8\alpha(F,0)}}{4\gamma(F,0)}
.
\]
Thus, for $\ze = (x + Z)/\left\|x+Z\right\|$, we can bound
\[
d_T(\zeta,x)
\le \| Z - X_0\| \le
\sigma(\al_0(f,x)) \beta(F,0) \le \sigma(\al) \beta_0(f,x)
.
\]
Again by Theorem 1, p.462, of \cite{Bez1},
\[
\|Z - X_1\| \le \frac{1-3\alpha(F,0)-\sqrt{(1+\alpha(F,0))^2-8\alpha(F,0)}}{4\gamma(F,0)}
\]
which implies that
\[
d_R(\zeta,y)
\le
(\sigma(\al)-1) \beta_0(f,x).
\]
\medskip

Let us prove that $x$ is an approximate zero.
This will follow directly from  \cite{BCSS}, Chap. 14, Theorem 1. In order
to apply this Theorem, we have to check its hypothesis
$$d_T(\ze,x)\ga_0(f,\ze) \leq \frac{3-\sqrt{7}}{2}. $$
Using Lemma \ref{high-der} this is obtained from
\begin{equation}\label{alpha:necessary1}
d_T(\ze,x)\frac{D^{3/2}}{2}\mu(f,\ze)\leq \frac{3-\sqrt{7}}{2}.
\end{equation}
We notice that
$u=
\frac{D^{3/2}}{2}d_R(x,\ze)\mu(f,x)
\le
u_T=\frac{D^{3/2}}{2}d_T(x,\ze)\mu(f,x)
\le
\alpha \sigma(\alpha) \le 0.049 \sigma(0.049) =0.0518\cdots < 1/19.$ According to Lemma~\ref{lem-9},
\[
\mu(f,\zeta)
\le 1.2 \mu(f,x) .
\]
Hence, we infer \eqref{alpha:necessary1} from:
\[
d_T(\ze,x)\frac{D^{3/2}}{2}\mu(f,\ze)\leq
1.2 u_T   < 0.1771\cdots = \frac{3-\sqrt{7}}{2}.
\]
A similar argument holds to prove that $y$ is an approximate root:
\[
d_T(\ze,y)\frac{D^{3/2}}{2}\mu(f,\ze)\leq
1.2 d_T(\ze,y)\frac{D^{3/2}}{2}\mu(f,x) \leq
1.2 \alpha(\sigma(\alpha) - \alpha)\leq
\frac{3-\sqrt{7}}{2}.
\]
\end{proof}
\medskip

In the following proposition we relate the invariant $\beta_0(f,x)$ for an approximate zero $x$ to the distance from its associated zero $\ze$.

\begin{proposition} \label{prop:beta-dist} Let $f \in \Hd$ be fixed
and $x \in \C^{n+1}$ be given.
If
\[
\frac{D^{3/2}}{2}\beta_0(f,x) \mu(f, x) \le \al \le 0.049,
\]
then,
$\beta_0(f,x) \le 1.128 d_{T}(x, \zeta)$,
where $\zeta$ is the zero of $f$ associated to $x$, given by
Theorem~\ref{the-3}.
\end{proposition}

\begin{proof} We suppose that both $f$, $x$, and $\ze$ are normalized.
From Theorem~\ref{the-3}, $ d_T(x, \zeta) \le \beta_0(f,x) \sigma(\al)$.
Hence,
\[
u \le u_T = \frac{D^{3/2}}{2} \mu(f,x) d_T(x, \zeta) \le \al \sigma(\al) \le 0.0518.
\]
From Theorem~\ref{the-1}, we conclude for later use that
\[
u_{T,\zeta} =
\frac{D^{3/2}}{2} \mu(f,\zeta) d_T(x, \zeta) \le
u (1+3.805u)
\]
Now we can bound:
\begin{eqnarray*}
\beta_0(f,x) &=&
\| Df(x)|_{x^\perp}^{-1} f(x) \|
\\
&\le&
\| Df(x)|_{x^\perp}^{-1} Df(\zeta)_{\zeta^\perp} \|
\| Df(\zeta)|_{\zeta^\perp}^{-1} f(x) \|
\\
&\le&
\left( \frac{1}{(1-u)^2} + u \right)
\| Df(\zeta)|_{\zeta^\perp}^{-1} f(x) \|
\end{eqnarray*}
using Lemmas~\ref{lem-3} and~\ref{lem-5}.
We further bound, as usual,
\[
\| Df(\zeta)|_{\zeta^\perp}^{-1} f(x) \|
\le
\|x-\zeta\| + \sum_{k \ge 2}
\frac{1}{k!}\left\|
Df(\zeta)|_{\zeta^\perp}^{-1}
D^kf(\zeta) (x-\zeta)^k
\right\|
\le
d_T(x,\zeta) \frac{1}{1-u_{T\zeta}}
.
\]
Putting all together,
\[
\beta_0(f,x) \le \frac{1 - u + u^2}{(1-u)^2} \frac{1}{1 - u(1+3.805u)}
d_T(x,\zeta)
\le
1.128
d_T(x,\zeta);
\]
the last step is obtained numerically, using $u \le \alpha \sigma(\alpha) \le 0.0518$.
\end{proof}

\begin{proposition} \label{prop-beta-of-iterate} Assume that
\[
\frac{D^{3/2}}{2} \beta_0(f,x) \mu(f,x) \le a \le 1/20
\]
Let $y=N_f(x)$. Then,
\[
\beta_0(f,y) \le
\frac{a (1-a)}{\psi(a)}
\left(1 + \frac{a}{1 - 3.805 a}\right)
\beta_0(f,x)
<
1.23 a \beta_0(f,x).
\]
\end{proposition}

\begin{proof}
We assume $\|x\|=1$.
Let $F: X \in x^{\perp} \mapsto f(x+X)$ be the affine polynomial
system associated with $f$.
Then, $\beta(F,0)=\beta_0(f,x)$. Moreover, we can scale
$y=x+Y$, for $Y = N_F(0)$. By Proposition 3, p.478 in
\cite{Bez1},
\begin{equation}\label{eq-beta-1}
\beta(F,Y) \le \frac{a (1-a)}{\psi(a)} \beta_0(f,x) .
\end{equation}
Moreover,
\begin{equation}\label{eq-beta-2}
\beta_0(f,y) \le
\| Df(y)|_{y^\perp} ^{-1}  Df(y)|_{x^\perp} \|
\|y\|^{-1}
\beta(F,Y)
.
\end{equation}
Clearly, $\|y\| \ge 1$. It remains to bound the norm of
the first term in the rhs of \eqref{eq-beta-2}.
By hypothesis,
\[
u= \frac{D^{3/2}}{2} \mu(f,x) d_R(f,x) \le a \le 1/20
\]
so Theorem~\ref{the-1} implies that
\[
\mu(f,y) \le \frac{1}{1-3.805a} \mu(f,x) .
\]
In particular, $Df(y)|_{y\perp}$ has full rank, and we can
apply Lemma~\ref{lem-6} and then Lemma~\ref{lem-5} to bound
\[
\| Df(y)|_{y^\perp} ^{-1}  Df(y)|_{x^\perp} \|
\le
1 + \delta(f,y) d_P(x,y)
\le
1 + D^{1/2} \mu(f,y) \beta_0(f,x)
\le 1 + \frac{a}{1 - 3.805 a}
\]
Combining with \eqref{eq-beta-1} and \eqref{eq-beta-2}, we
obtain:
\[
\beta_0(f,y) \le
\frac{a (1-a)}{\psi(a)}
\left(1 + \frac{a}{1 - 3.805 a}\right)
\beta_0(f,x) .
\le
1.23 a \beta_0(f,x)
\]
\end{proof}

%
%
%
%
%

\section{The homotopy}\label{sec-homotopy}

The objective of this section is to prove Theorem~\ref{main-homotopy}.
Through this section the considered systems and zeros are normalized: $f(\ze)=0$ with $\left\|f\right\| = \left\|\ze\right\|=1$. 

Let $t \in [a,b]$, and $x_t \in \C^{n+1}$ be given with $\left\| x_t \right\|=1$. We suppose that
\begin{equation}\label{E}
\rank Df_t(x_t)|_{x_t^\perp} = n,
\end{equation}
and that
\begin{equation}\label{E0}
\frac{D^{3/2}}{2}\be_0(f_t,x_t)\mu(f_t,x_t) \leq \al .
\end{equation}
According to Theorem \ref{the-3}, for $\al$ small enough,
$x_t$ is an approximate zero of $f_t$. We call $\zeta_t$ the associated zero
and extend it continuously for $s \in [t,t']$ so that $f_s(\zeta_s) = 0$.

The main difficulty to prove Theorem~\ref{main-homotopy}
is to transfer the properties \eqref{E} and \eqref{E0}
supposed to be true at $t=t_i$ onto a similar property at $t'=t_{i+1}$.
Moreover, we must show that if $x_t$ is an approximate zero associated to
$\zeta_t$, then the same is true for $t'$, for a continuous path $\zeta_s$.
For this purpose we study this transfer in a general context.

Through this section, $\epsilon \le 1/6$.
Let $t' > t$ be given and assume that
\begin{equation}\label{E1}
l(t') - l(t) \le \frac{\epsilon}{5 D^{1/2} \mu(f_t,x_t)}
\end{equation}
and
\begin{equation}\label{E2}
\max_{s \in [t, t']} \phi_{t,s} (x) \le \frac{\epsilon}{5 D^{3/2} \mu(f_t,x_t)}.
\end{equation}

For any $s \in ]t,t']$ let us define $x_s = N_{f_s}(x_t)$. Notice that $x_s$ is not necessarily normalized.

\begin{lemma}\label{homotopy-lemma}
Let $\epsilon \le 1/6$ and set $\alpha = \epsilon^2/2$.
Under the hypotheses above, for any $s \in ]t, t']$, one has
\begin{enumerate}
\item
$\mu(f_s, x_t) \le \frac{1}{1-\epsilon} \mu(f_t,x_t)$.
\item $\frac{1}{1+\epsilon/5}\left( \phi_{t, s} (x_t)  - \frac{2 \alpha }
{D^{3/2} \mu(f_t,x_t)}\right)
\le \beta_0(f_s, x_t) \le \frac{\epsilon/5+2 \al}{ (1-\epsilon/5)D^{3/2}\mu(f_t,x_t)}$,
\item $\frac{D^{3/2}}{2}\beta_0(f_s, x_t)\mu(f_s, x_t) \le 0.049$.
In particular,
$x_t$ and $x_s$ are approximate zeros of $f_s$ associated with $\ze_s$,
\item
$\mu(f_s, x_s) \le \frac{1}{1-\epsilon} \mu(f_t,x_t)$.
\item $
(1-\epsilon) \mu(f_s, \zeta_s)
\le
\mu(f_t, x_t)
\le
(1+\epsilon) \mu(f_s, \zeta_s)
$.
In particular,
$\zeta_s$ is non-degenerate zero of $f_s$, and hence $s \mapsto \zeta_s$
in continuous for $s \in [t,t']$.
\item $\beta_0(f_s,x_s) \le 1.23 \alpha(f_s, x_t) \beta_0(f_s,x_t).$

\item Hypothesis \eqref{E} and a strong version of \eqref{E0} hold at $s$: $\rank Df_s(x_s)|_{x_s^\perp} = n$, and
\[\frac{D^{3/2}}{2}\be_0(f_s,x_s)\mu(f_s,x_s) \leq 0.128\al
\]
\end{enumerate}
\end{lemma}

\begin{proof}
\noindent {\bf 1.}
From equation \eqref{E1},
\[
\|f_t - f_s\| \le l(s)-l(t) \le l(t')-l(t) \le \frac{\epsilon}{5 D^{1/2} \mu(f_t,x_t)},
\]
so that $v=\sqrt{D} \mu(f_t,x_t) \|f_t-f_s\| \le \epsilon/5 \le 1/20$,
and Corollary~\ref{cor-1} gives $\left(1 - \epsilon\right)\mu(f_s, x_t) \le \mu(f_{t}, x_t)$.

\medskip

\noindent {\bf 2.}
$$
\beta_0(f_s, x_t) = \| \invder{f_s}{x_t} f_s(x_t) \| \le $$
$$\| \invder{f_s}{x_t} Df_{t}(x_t)|_{x_t^{\perp}} \| \left( \| \invder{f_{t}}{x_t} \left(f_s(x_t) - f_t(x_t)\right)\|
+ \| \invder{f_t}{x_t} f_t(x_t) \|\right).$$
By Lemma \ref{lem-10}, \eqref{E0} and \eqref{E2} we obtain:
\[
\beta_0(f_s, x_t) \leq \frac{1}{1-v}\left( \phi_{t, s} (x_t)  + \beta_0(f_t,x_t)\right) \le
\]
\[\frac{1}{1-\epsilon/5} \left( \frac{\epsilon}{5 D^{3/2}\mu(f_t,x_t)} + \frac{2\al}{ D^{3/2}\mu(f_t,x_t)} \right) =
\frac{\epsilon/5+2 \al}{ (1-\epsilon/5)D^{3/2}\mu(f_t,x_t)}.
\]
\medskip
For the lower bound,
$$
\beta_0(f_s, x_t) = \| \invder{f_s}{x_t} f_s(x_t) \| \ge $$
$$
\left\| \left(\invder{f_s}{x_t} Df_{t}(x_t)|_{x_t^{\perp}}\right)^{-1}\right\|
^{-1}
\left( \| \invder{f_{t}}{x_t} \left(f_s(x_t) - f_t(x_t)\right)\|
- \| \invder{f_t}{x_t} f_t(x_t) \|\right).$$
By Lemma \ref{lem-10}, \ref{E0} and \ref{E2} we obtain:
\[
\beta_0(f_s, x_t) \ge \frac{1}{1+v}\left( \phi_{t, s} (x_t)  - \beta_0(f_t,x_t)\right) \ge
\]
\[
\ge \frac{1}{1+\epsilon/5}\left( \phi_{t, s} (x_t)  - \frac{2 \alpha }
{D^{3/2} \mu(f_t,x_t)} \right).
\]
\medskip

\noindent {\bf 3.} Combining the the two preceding items,
\begin{equation}\label{item-3}
\frac{D^{3/2}}{2} \beta_0(f_s, x_t)\mu(f_s,x_t)
\le
\frac{\epsilon/10+\al}{(1-\epsilon) (1-\epsilon/5) }
\le
0.037931\cdots < 0.049.
\end{equation}
Thus, by Theorem \ref{the-3}, $x_t$  and $x_s$ are
approximate zeros of $f_s$ associated with $\ze_s$.

\medskip
\noindent {\bf 4.}
Using item 2,
\[
d_R(x_t, x_s) \le \beta0(f_s, x_t) \le
\frac{\epsilon/5+2 \al}{ (1-\epsilon/5)D^{3/2}\mu(f_t,x_t)}.
.
\]
Thus,
\[
u=D^{3/2} d_R(x_t, x_s) \mu(f_t,x_t) /2
\le
\frac{\epsilon/10+ \al}{ (1-\epsilon/5)} < \epsilon/5 .
\]
Then, we can use Corollary~\ref{cor-1} again to bound
\[
\mu(f_s, x_s) \le \frac{1}{1-\epsilon} \mu(f_t,x_t) .
\]

\medskip
\noindent {\bf 5.}
From Theorem~\ref{the-3},
\[
d_R(x_t,\zeta_s) \le d_T(x_t, \zeta_s) \le
\sigma(\alpha(f_s,x_t)) \beta_0(f_s, x_t)
\le
1.0429\cdots  \beta_0(f_s, x_t).
\]
Thus,
\[
u = \frac{D^{3/2}}{2} \mu(f_t,x_t) d_R(x_t, \zeta_s)
\le
0.1978026 \cdots \epsilon < \epsilon/5
.
\]
We bounded in item 1 the quantity $v= D^{1/2} \mu(f_t,x_t) \|f_t-f_s\| <
\epsilon/5$. Hence, by Corollary~\ref{cor-1} again:
\[
(1-\epsilon) \mu(f_s, \zeta_s) \le \mu(f_t, \zeta_t)
\le (1+\epsilon) \mu(f_s, \zeta_s) .
\]
\medskip
\noindent {\bf 6.}
From item 3 and Proposition~\ref{prop-beta-of-iterate},
\[
\beta_0(f_s, x_s) \le
1.23
\alpha_0(f_s,x_t)
\beta_0(f_s,x_t)
.
\]
\medskip

\noindent {\bf 7.}
Because $\mu(f_s,x_s)$ is finite, $Df_s(x_s)$ has full rank.
From items 1, 6 and \eqref{item-3},
\[
\frac{D^{3/2}}{2}
\beta_0(f_s,x_s)
\mu(f_s, x_s)
\le
1.23
\left( \frac{\epsilon/10+\alpha}{(1-\epsilon)(1-\epsilon/5)}
\right)^2
\le
0.128
\alpha.
\]
\end{proof}

Recall that our algorithm allows for an approximate computation
of the Newton iteration. The robustness Lemma below shows that if a
point $x$ satisfies \eqref{E0} and conclusion 7 of
Lemma~\ref{homotopy-lemma}, then an approximation $y$ of $x$
satisfies \eqref{E0} and \eqref{E1}.

\begin{lemma}\label{lem-robustness} Assume that $\|f\|=1$ and
$\|x\|=\|y\|=1$. Let $\alpha \le 1/72$ and $c \le 0.8$.
Suppose that $Df(x)_{|x^{\perp}}$ has rank $n$, and
\begin{eqnarray}
\label{R-H1}
\frac{D^{3/2}}{2} \beta_0(f,x) \mu(f,x) \le 0.128 \alpha \\
\label{R-H2}
u=\frac{D^{3/2}}{2} d_R(x,y) \mu(f,x) \le \frac{c \alpha}{\sqrt{D} \mu(f,x)}
.
\end{eqnarray}
Then, $Df(y)_{|y^{\perp}}$ has rank $n$, and
\begin{equation}
\label{R-C1}
\frac{D^{3/2}}{2} \beta_0(f,y) \mu(f,y) \le \alpha \\
\end{equation}
and furthermore, $x$ and $y$ are approximate zeros associated to
the same exact zero $\zeta$.
\end{lemma}

\begin{proof}
By using $D^{3/2} \mu(f,x) \ge 4$ (see Lemma~\ref{lem-1} and the hypothesis $D \geq 2$)
we obtain that $u \le 0.0055\cdots < 1/19$.
Therefore, Lemma~\ref{lem-9} implies that
\[
\mu(f,y) \le (1+3.805) \mu(f,x)
\]
and in particular, $Df(y)_{|y^{\perp}}$ has rank $n$.
\medskip
\par
To estimate $\beta_0(f,y)$, we decompose
\[
\beta_0(f,y) = 
\left\|
Df(y)_{|y^{\perp}}^{-1}
f(y)
\right\|
\le
\left\|
Df(y)_{|y^{\perp}}^{-1}
Df(x)_{|x^{\perp}}
\right\|
\left\|
Df(x)_{|x^{\perp}}^{-1}
f(y)
\right\|
.
\]
The first term is bounded by Lemma~\ref{lem-7},
\[
\left\|
Df(y)_{|y^{\perp}}^{-1}
Df(x)_{|x^{\perp}}
\right\|
\le
1+3.805u
.
\]
Taylor's exansion gives $ Df(x)_{|x^{\perp}}^{-1} f(y) = $ 
\[
Df(x)_{|x^{\perp}}^{-1} f(x)
+
Df(x)_{|x^{\perp}}^{-1} 
Df(x) (y-x)
+
\sum_{k \ge 2} \frac{1}{k!}
Df(x)_{|x^{\perp}}^{-1} D^kf(x)(y-x)^k
.
\]
Taking norms,
\[
\left\| Df(x)_{|x^{\perp}}^{-1} f(y) \right\|
\le
\beta_0(f,x)
+
\delta(f,x) \|y-x\|
+
\frac{\|y-x\|^2 \gamma_0(f,x)}{1-
\|y-x\| \gamma_0(f,x)}
.
\]
By Lemma~\ref{high-der}c,$\|y-x\| \gamma_0(f,x) \le u$. Hence,
\[
\beta_0(f,x) \le 
(1+3.805 u) 
\left(
\beta_0(f,x)
+
\delta(f,x) \|y-x\|
+
\frac{\|y-x\| u}{1-u}
\right)
.
\]
Using Lemma~\ref{lem-5}, $\delta(f,x) \le \sqrt{D}\mu(f,x)$. Thus,
\begin{eqnarray*}
\frac{D^{3/2}}{2} \beta_0(f,y) \mu(f,y) &\le& 
(1+3.805 u)^2
\left( 0.128 \alpha + \sqrt{D} \mu(f,x) u + \frac{u^2}{1-u} \right) \\
&\le&
(1+3.805 u)^2
\left( 0.128 + c + \frac{cu}{\sqrt{D}\mu(f,x) (1-u)} \right) \alpha \\
&\le& 0.97 \alpha < \alpha
\end{eqnarray*}
Since $\alpha \le 1/72 \le 0.049$, Theorem~\ref{the-3} implies that
both $x$ and $y$ are approximate zeros of $f$. As this is also the
case for all the points in the shortest arc of circle between $x$
and $y$, the associated zero must be the same.
\end{proof}

\begin{lemma}\label{alt1}
Assume the Hypotheses of Lemma~\ref{homotopy-lemma}.
Let $0 < \xi \le 1$. If furthermore
\[
l(t')-l(t) \ge \frac{\xi \epsilon}{5 \sqrt{D} \mu(f_t,x_t)}
,
\]
then
\[
L(t') - L(t) \ge
\frac{ \xi \epsilon} {5 D^{3/2} } .
\]
\end{lemma}

\begin{proof}
\begin{eqnarray*}
L(t') - L(t)
&=&
\int_{t}^{t'}
\left\| (\dot f_s, \dot \zeta_s) \right\| \mu(f_s, \zeta_s)
\ ds
\\
&\ge&
\int_{t}^{t'}
\left\| \dot f_s \right\| \mu(f_s, \zeta_s)
\ ds
\\
&\ge&
\frac{\mu(f_t,x_t)}{1+\epsilon}
\int_{t}^{t'}
\left\| \dot f_s \right\|
\ ds \hspace{1em} \text{using Lemma~\ref{homotopy-lemma}(5).}
\\
&=&
\frac{\mu(f_t,x_t)}{1+\epsilon}
(l(t')-l(t))
\\
&\ge&
\frac{ \xi \epsilon }{5 (1+\epsilon) \sqrt{D}}.
\\
&\ge&
\frac{ \xi \epsilon} {5 D^{3/2} } .
\end{eqnarray*}
\end{proof}

\begin{lemma}\label{alt2}
Assume the Hypotheses of Lemma~\ref{homotopy-lemma}
and choose $\epsilon$ and $\xi$ so that $20 \epsilon \le \xi \le 1$.
If furthermore
\[
\phi_{t, t'}(x) \ge \frac{\xi \epsilon}{ 5 D^{3/2} \mu(f_t,x_t)},
\]
then
\[
L(t') - L(t) \ge
\frac{ \xi \epsilon} {13 D^{3/2} }.
\]
\end{lemma}

\begin{proof}
\begin{eqnarray*}
L(t') - L(t)
&=&
\int_{t}^{t'}
\left\| (\dot f_s, \dot \zeta_s) \right\| \mu(f_s, \zeta_s)
\ ds
\\
&\ge&
\int_{t}^{t'}
\left\| \dot \zeta_s \right\| \mu(f_s, \zeta_s)
\ ds
\\
&\ge&
\frac{\mu(f_t,x_t)}{1+\epsilon}
\int_{t}^{t'}
\left\| \dot \zeta_s \right\|
\ ds \hspace{1em} \text{using Lemma~\ref{homotopy-lemma}(5).}
\\
&\ge&
\frac{\mu(f_t,x_t)}{1+\epsilon}
d_R(\zeta_{t}, \zeta_{t'}).
\end{eqnarray*}
By the triangle inequality,
\[
d_R(\zeta_{t}, \zeta_{t'})
\ge
d_R(\zeta_{t'},x_t)
-
d_R(x_t, \zeta_{t})
\]
We know from Theorem~\ref{the-3} that $d_R(x_t,\zeta_t) \le \sigma(\alpha) \beta_0(f_t, x_t)$.
Lemma~\ref{homotopy-lemma}(3) says that $\alpha_0(f_{t'},x_{t'}) \le 0.049$
and hence, from Proposition~\ref{prop:beta-dist}, we obtain:
\[
d_T(x_t, \zeta_{t'}) \ge \beta_0(f_{t'},x_t) / 1.128 .
\]
Thus, $d_T(x_t, \zeta_{t'}) \ge
\frac{ \xi \epsilon/5 - 2 \alpha }{(1+\epsilon/5)D^{3/2} \mu(f_t, x_t)}
\frac{1}{1.128}$ by Lemma~\ref{homotopy-lemma}(2).
We use now the bound $20 \epsilon \le \xi$ and the fact that $\alpha=\epsilon^2/2$ to deduce that
\[
d_T(x_t, \zeta_{t'}) \ge
\omega = \frac{ 3 \xi \epsilon /20 }{1.128 (1+\epsilon/5)D^{3/2} \mu(f_t, x_t)}
\]
Since $\epsilon \le \xi/20 \le 1/20$, we can bound $\omega \le 0.001645\cdots$.
Of course,
$ d_R(x_t, \zeta_{t'}) \ge \arctan \omega$.
We may bound $\arctan(\omega) \ge \omega
\arctan'(0.001645\cdots) = \frac{1}{1+0.001645\cdots^2} \omega> 0.999 \omega$.
\medskip
Now,
\begin{eqnarray*}
L(t')-L(t) &\ge&
\frac{ \mu(f_t, x_t)}{1+\epsilon}
\left( 0.999 \frac{3\xi \epsilon/20}{1.128 (1+\epsilon/5)D^{3/2}
\mu(f_t,x_t)}
- \frac {2 \sigma(\alpha) \alpha}{D^{3/2} \mu(f_t,x_t)}
\right)
\\
&\ge&
0.0775\frac{ \xi \epsilon }{D^{3/2}}
>
\frac{ \xi \epsilon }{13 D^{3/2}}
\end{eqnarray*}
\end{proof}

\noindent{\bf Proof of Theorem~\ref{main-homotopy}}
\vskip 3mm
We take $\xi = 20\epsilon$ and $\alpha = \epsilon^2/2$.
Assume that $L(b)$ is finite. By hypothesis and Theorem~\ref{the-3},
$x_0$ is an approximate zero of $f_a$. $(f_t, \zeta_t)_{t \in [a,b]}$
denotes the unique lifting of the path $f_t$ corresponding to 
$\zeta_0$ zero of $f_a$ associated to $x_0$.

\noindent
{\bf Induction hypothesis:}{\em
\[
\frac{D^{3/2}}{2} \mu(f_{t_i} , x_i) \beta_0(f_{t_i},x_i) < \alpha
\]
and furthermore, $x_i$ is an approximate zero associated to
$\zeta_{t_i}$.}

The induction hypothesis holds by hypothesis at $i=0$, so we assume it
is verified up to step $i$.
We are in the hypotheses of Lemma~\ref{homotopy-lemma}
for $t=t_i$, $t'=t_{i+1}=\min(s,s',b)$ and $x=x_{t_i} = x_i$.
Thus, 
\[
\frac{D^{3/2}}{2} 
\mu\left(f_{t_{i+1}} , N_{f_{t_i}}(x_i)\right)
\beta_0\left(f_{t_{i+1}} , N_{f_{t_i}}(x_i)\right)
< 0.128 \alpha
\]
and $N_{f_{t_i}}(x_i)$ is an approximate zero for $f_{t_{i+1}}$ associated
to $\zeta_{t_{i+1}}$.

Now we are in the hypotheses of Lemma~\ref{lem-robustness}. Thus,
$y=x_{i+1}$ picked at \eqref{s3} satisfies the induction hypothesis.
\medskip
\par

In order to bound
the number of iterations, we remark that at each step $i$, one
of the following alternatives is true:
\begin{enumerate}
\item This is the last step: $t_{i+1}=b$.
\item Condition \eqref{s1} is true. In that case, we are under the hypotheses
of Lemma~\ref{alt1}.
\item Condition \eqref{s2} is true. Then we are under the hypotheses of
Lemma~\ref{alt2}.
\end{enumerate}
Therefore, we may infer that at each non-terminal step,
\[
L(t_{i+1}) - L(t_i) > \frac{\xi \epsilon}{13 D^{3/2}}
.\]
Therefore, there can be no more than $260 \xi^{-2} L(b) D^{3/2}$ non-terminal
steps. There is only one terminal step, so the total number of steps
is at most
\[
1+260 \xi^{-2} L(b) D^{3/2} = 1+0.65\ep^{-2} L(b) D^{3/2}
\]

\end{document}